\documentclass[12pt]{article}
\usepackage{amssymb,url, amsthm, amsmath, verbatim}
\usepackage[koi8-r]{inputenc}
\usepackage{amsmath}
\usepackage{amsfonts}
\usepackage{amssymb}
\usepackage{mathrsfs}
\usepackage{amsthm}
\usepackage{hyperref}
\usepackage{color}

\usepackage[russian,english]{babel}

\newtheorem{theorem}{Theorem}
\newtheorem{lem}{Lemma}
\newtheorem{prop}{Proposition}
\newtheorem{remark}{Remark}

\newcommand{\eps}{\varepsilon}

\author{A.I.Nazarov, D.M.Stolyarov, and P.B.Zatitskiy}
\title{Tamarkin equiconvergence theorem\\ and trace formula revisited}

\newcommand{\Span}{{\bf Span}}
\newcommand{\tr}{{\bf tr}}
\newcommand{\OOO}{\textup{O}}
\newcommand{\UUU}{\textup{u}}
\newcommand{\Sp}{{\bf Sp}}
\newcommand{\ff}{f}
\newcommand{\DD}{\Xi}
\newcommand{\I}{I}
\renewcommand{\leq}{\leqslant}

\renewcommand{\geq}{\geqslant}
\renewcommand{\ge}{\geqslant}

\begin{document}
\maketitle
\abstract{We obtain a simple formula for the first-order trace of a regular differential operator on a segment perturbated by a multiplication operator. The main analytic ingredient of the proof is an improvement of the Tamarkin equiconvergence theorem.}

\tableofcontents
\selectlanguage{english}

\section{Introduction}
\subsection{Historical remarks}
Consider a formal differential expression of order $n\ge2$,
\begin{equation}\label{operator}
\ell :=(-i)^n D^{n} + \sum\limits_{k = 0}^{n-2} p_k(x) D^k,
\end{equation}
acting on functions on some segment $[a,b]$ ($D$ denotes differentiation in $x$). We assume $p_k$ to be summable functions. Let $P_j$ and $Q_j$, 
$j \in \{0,\dots,n-1\}$, be polynomials whose degrees do not exceed $n-1$. Then one can form the boundary conditions:
\begin{equation}\label{bc}
P_j(D)y(a) + Q_j(D)y(b) = 0,\qquad j\in\{0,\dots,n-1\},
\end{equation}
where $y$ is an arbitrary function.

Let $d_j$, $j \in \{0,\dots,n-1\}$, be the maximum of degrees of $P_j$ and $Q_j$. Suppose $a_j$ and $b_j$ are the $d_j$-th coefficients of $P_j$ and $Q_j$ 
respectively. We assume that the system of boundary conditions (\ref{bc}) is normalized, i.e. $\sum\limits_j d_j$ is minimal among all the systems of 
boundary conditions that can be obtained from (\ref{bc}) by linear bijective transformations. See \cite[Ch. II, \S4]{Nay} for a detailed explanation and 
\cite{Shk} for a more advanced treatment. We call the system (\ref{bc}) \emph{almost separated} if after some permutation of the boundary 
conditions we have
$$\aligned
&\mbox {for}\ \ n=2m: && b_j = 0\ \ \mbox{if}\ \ j < m; \quad a_j = 0\ \ \mbox{if}\ \ j\geqslant m; \\
&\mbox {for}\ \ n=2m+1: && b_j = 0\ \ \mbox{if}\ \ j < m; \quad a_j = 0\ \ \mbox{if}\ \ j> m; \quad a_mb_m\ne0. 
\endaligned
$$

The differential expression \eqref{operator} and the boundary conditions \eqref{bc} generate an 
operator~${\mathbb L}$ (see~\cite[Ch. I]{Nay} for this standard procedure). 
We assume these boundary conditions to be Birkhoff regular (see \cite[Ch. II, \S4]{Nay}). 
We underline that we do not require our operator to be self-adjoit;
in particular, all the coefficients may be non-real.

We note that the operator ${\mathbb L}$ has purely discrete spectrum (see \cite[Ch. I]{Nay}) and denote it by $\{\lambda_N\}_{_{N=1}}^{^\infty}$. We always 
enumerate the points of a spectrum in ascending order of their absolute values according to the multiplicity of eigenvalues, e.g., we assume 
that $|\lambda_N| \leqslant |\lambda_{N+1}|$.
 
Let ${\mathbb Q}$ be an operator of multiplication by a function $q \in L^1([a,b])$.  
Then, ${\mathbb L} + {\mathbb Q}$ also has purely discrete spectrum $\{\mu_N\}_{_{N=1}}^{^\infty}$.

In the previous paper \cite{ZNS}, the authors obtained a formula for regularized trace  
\begin{equation}\label{trace}
\sum\limits_{N = 1}^{\infty} (\mu_N - \lambda_N)
\end{equation}
in terms of degrees of $P_j$ for the case of a self-adjoint semibounded operator with discrete spectrum on the halfline $\mathbb{R}_+$ 
(the above series converges iff $\int q = 0$, see Theorem~1 in~\cite{SKP}, 
otherwise one has to regularize the trace to get something worth counting). Partial cases of this problem were considered earlier in papers \cite{KP1}, 
\cite{SKP}, \cite{KP3}, and in our preprint \cite{PoSledam}. 

We conjectured that a similar formula should be valid for the case of an interval at least if the boundary conditions are almost separated. 
This is really the case, though the details are dramatically different. In \cite{ZNS} we used the theorem on asymptotic behavior of the spectral functions 
of ${\mathbb L}$ and of the operator generated by the truncated expression and the same boundary conditions \eqref{bc} obtained in \cite{Ko1}, \cite{Ko2}. 
Surprisingly, for the case of an interval the corresponding result was not known yet! So we had to prove this theorem, which refines the classical 
equiconvergence result of Tamarkin (see \cite{T} or Theorem~$1.5$ in \cite{Minkin}).

Theory of regularized traces originated in the 50-th. We refer the reader to the survey \cite{SadPod} for the historical scenery of the subject in general. 
We mention only several results that our one generalizes. The first paper where such problems were considered was \cite{GL}, the formula of regularized 
trace was calculated for the perturbation of a self-adjoint second order operator by a multiplication operator. Some particular cases of fourth order 
operators were treated in \cite{HK}, \cite{LS} and \cite{turkish}. Operators of an arbitrary order without lower-order coefficients was 
considered in \cite{Sh}. A formula for regularized trace was obtained for general Birkhoff regular boundary conditions. However, we should mention 
that the paper \cite{Sh} deals with the case of a more regular function $q$ and does not provide short answers for the cases of almost separated and 
quasi-periodic boundary conditions. A special case of boundary conditions (all derivatives of even order vanish on both ends of the interval) for 
self-adjoint operators of even order with lower-order coefficients was considered in \cite{Sad2}, where formulas for ${\cal S}(q)$ and for traces of 
higher order were given in terms of zeta function.

\subsection{Setting of the problem and formulation of results}

Let ${\mathbb L}_0$ be the operator generated by the differential expression $(-i)^n D^{n}$ and the boundary conditions \eqref{bc}. Denote by 
$\{\lambda^0_N\}_{_{N=1}}^{^\infty}$ the eigenvalues of ${\mathbb L}_0$. Consider also the Green functions of operators ${\mathbb L}_0 - \lambda$ and ${\mathbb L} - \lambda$, 
which we denote by $G_0(x,y,\lambda)$ and $G(x,y,\lambda)$, respectively. Then our main estimate reads as follows.

\begin{theorem}\label{Spint}
For every sequence $R=R_l \rightarrow \infty$ separated from $|\lambda^0_N|^{\frac{1}{n}}$ the integral 
\begin{equation*}
\int\limits_{|\lambda|=R^{n}} |(G_0 - G)(x,y,\lambda)|\, |d\lambda|  
\end{equation*}
tends to zero uniformly in $x,y\in [a,b]$.
\end{theorem}

This theorem is a generalization of the celebrated Tamarkin equiconvergence theorem mentioned above. Denote by $\theta_R(x,y)$ the integral 
$\int\limits_{|\lambda|=R^{n}} (G_0 - G)(x,y,\lambda)\, d\lambda$. Then the Tamarkin theorem states that the integral operator with the kernel 
$\theta_R$ considered as an operator from $L^1$ to $L^{\infty}$ tends to zero in the strong operator topology. Theorem~\ref{Spint} implies the same 
convergence in the norm operator topology. Though we found this theorem during our study of regularized traces, it is interesting in itself.\medskip 

Now we turn to traces. Unfortunately, a beautiful formula as the one we had in~\cite{ZNS} does not hold for the general problem. So, we need to introduce
some notation.

Let $\nu_1=[\frac{n+1}{2}]$ and $\nu_2=[\frac{n}{2}]$. For $\kappa=1,2$ denote by $\hat {\cal W}^{[\kappa]}$ the matrix
\begin{equation}\label{hat}
\hat {\cal W}^{[\kappa]}=
\begin{pmatrix}
a_0 & \dots& \rho^{(\nu_\kappa-1)d_0} a_{0}& \rho^{\nu_\kappa d_0} b_{0}&\dots &\rho^{(n-1)d_0} b_{0}\\
\vdots& &\vdots &\vdots& &\vdots\\
a_{n-1} & \dots& \rho^{(\nu_\kappa-1)d_{n-1}} a_{n-1}& \rho^{\nu_\kappa d_{n-1}} b_{n-1}&\dots &\rho^{(n-1)d_{n-1}} b_{n-1}\\
\end{pmatrix}
\end{equation}
(here and further $\rho = e^{\frac{2\pi i}{n}}$). Note that these matrices are non-degenerate by the Birkhoff regularity condition. 

Next, define matrices ${\cal A}$ and ${\cal B}$ with entries
$$
{\cal A}_{jk}=a_{j-1} (\rho^{k-1})^{d_{j-1}};\qquad {\cal B}_{jk}=b_{j-1} (\rho^{k-1})^{d_{j-1}} \quad \mbox{for} \quad j,k \in \{1,\dots,n\}. 
$$
Finally, we introduce matrices ${\cal P}^{[\kappa]}$ and ${\cal Q}^{[\kappa]}=(\bar{{\cal P}}^{[\kappa]})^T$, $\kappa=1,2$, by formulas
\begin{equation}\label{P}
{\cal P}^{[\kappa]}_{\alpha \beta} = \begin{cases}
\frac{1}{\rho^{\beta-\alpha}-1}, &\alpha>\nu_\kappa \geqslant \beta; \\
0, & \mbox{otherwise};
\end{cases}
\qquad
{\cal Q}^{[\kappa]}_{\alpha \beta} = \begin{cases}
\frac{1}{\rho^{\beta-\alpha}-1}, &\beta>\nu_\kappa \geqslant \alpha; \\
0, & \mbox{otherwise};
\end{cases}
\end{equation}
Note that if $n$ is even, then $\nu_1=\nu_2=\frac{n}{2}$, $\hat {\cal W}^{[1]}=\hat {\cal W}^{[2]}$, ${\cal P}^{[1]}={\cal P}^{[2]}$, and ${\cal Q}^{[1]}={\cal Q}^{[2]}$.\medskip 

Now we can formulate the main result of our paper.

\begin{theorem}\label{gentraceth}
Let $q\in L^1([a,b])$ be such that the functions
$$
\psi_a(x)=\frac{1}{x-a}\int\limits_a^x q(t)\, dt;\qquad
\psi_b(x)=\frac{1}{b-x}\int\limits_x^b q(t)\, dt
$$
have bounded variation at the points $a$ and $b$, respectively. Then for the eigenvalues $\lambda_N$ and $\mu_N$ of the operators ${\mathbb L}$ and ${\mathbb L}+{\mathbb Q}$ 
defined above the following is true:
\begin{multline}\label{mainform}
{\cal S}(q) \equiv \sum\limits_{N=1}^{\infty}\Big[\mu_N-\lambda_N-\frac{1}{b-a} \int\limits_{a}^{b}q(t)\, dt\Big] \\
= \frac{\psi_a(a+)}{2n}\cdot\sum\limits_{\kappa=1}^2\tr({\cal P}^{[\kappa]}(\hat {\cal W}^{[\kappa]})^{-1}{\cal A})
+\frac{\psi_b(b-)}{2n}\cdot\sum\limits_{\kappa=1}^2\tr({\cal Q}^{[\kappa]}(\hat {\cal W}^{[\kappa]})^{-1}{\cal B}).
\end{multline}
Moreover, for $\kappa=1$ and $\kappa=2$ the following formula is true:
\begin{equation}\label{sumcoeff}
\tr({\cal P}^{[\kappa]}(\hat {\cal W}^{[\kappa]})^{-1}{\cal A})+\tr({\cal Q}^{[\kappa]}(\hat {\cal W}^{[\kappa]})^{-1}{\cal B})=\sum\limits_{j=0}^{n-1}d_j-\frac{n(n-1)}{2}.
\end{equation}
\end{theorem}

\begin{remark}
Formula~$(\ref{mainform})$ for ${\mathbb L}={\mathbb L}_0$ and smooth $q$ was obtained in $\cite{Sh}$. However, formula~$(\ref{sumcoeff})$, as well as Theorem~\ref{sep} 
below, is new even in this case.
\end{remark}

For some classes of boundary conditions formula~(\ref{mainform}) can be considerably simplified.

\begin{theorem}\label{sep}
Let the assumptions of Theorem~\ref{gentraceth} be satisfied. 

{\bf 1}. Suppose that the boundary conditions \eqref{bc} are almost separated. Then

\noindent 1a$)$ for $n=2m$,
\begin{equation}\label{sep-even}
{\cal S}(q) = \frac{\psi_a(a+)}{2m}\left(\sum\limits_{j=0}^{m-1}d_j-\frac{m(2m-1)}{2}\right)+
\frac{\psi_b(b-)}{2m}\left(\sum\limits_{j=m}^{2m-1}d_j-\frac{m(2m-1)}{2}\right);
\end{equation} 

\noindent 1b$)$ for $n=2m+1$,
\begin{multline}\label{sep-odd}
{\cal S}(q) = \frac{\psi_a(a+)}{2m+1}\left(\sum\limits_{j=0}^{m-1}d_j+\frac{d_m}2-\frac{m(2m+1)}{2}\right)\\
+\frac{\psi_b(b-)}{2m+1}\left(\sum\limits_{j=m+1}^{2m}d_j+\frac{d_m}2-\frac{m(2m+1)}{2}\right).
\end{multline}

{\bf 2}. Suppose that the boundary conditions \eqref{bc} are quasi-periodic, i.e. $d_j=j$ and $b_j=a_j\vartheta$  $(\vartheta\ne0)$ for 
$j\in\{0,\dots,n-1\}$. Then 
\begin{equation}\label{periodic}
{\cal S}(q)  = 0.
\end{equation}
\end{theorem}

The plan of our paper is as follows. In Section~\ref{2} we prove Theorem~\ref{Spint}, almost by direct computation. Here we also establish auxiliary estimates to be used in the next section. In Subsection~\ref{3.1} we deduce 
formula~(\ref{mainform}) from Theorem~\ref{Spint}. To do this, we improve the idea of~\cite{Sh}. Finally, in Subsection~\ref{3.2} we derive formulas~(\ref{sumcoeff})--(\ref{periodic}) using similar technique and tricks to those we used in~\cite{ZNS}.

\section{Proof of Theorem~\ref{Spint}}\label{2}

Throughout the paper we use the following notation. For $\lambda \in \mathbb{C}$ we define $z=\lambda^{\frac 1n}$, ($Arg(z)\in [0,2\pi/n)$). 
For a function $\Phi$ defined on $\mathbb{C}$, we write $\tilde\Phi(z)=\Phi(\lambda)$.

\subsection{Formula for the Green function}

We begin with finding the exact value of $G_0$ (recall that this is the Green function of~${\mathbb L}_0 - \lambda$). We introduce a fundamental solution for the 
operator generated by $(-i)^nD^{n}-\lambda$:
\begin{equation*}
\Tilde{K}_0(x,y,z)= 
\begin{cases}
0,& a \leqslant x<y\leqslant b\\
\frac{i}{nz^{n-1}}\sum\limits_{k=0}^{n-1} \rho^{k}e^{iz\rho^{k}(x-y)},& a\leqslant y\leqslant x \leqslant b.
\end{cases}
\end{equation*}
We search $\Tilde{G_0}$ as 

\begin{equation*}
\Tilde{G_0}(x,y,z)=\Tilde{K_0}(x,y,z)-\frac{i}{nz^{n-1}}\sum\limits_{k=0}^{n-1} c_k(y,z) e^{iz \rho^k (x-y)}.
\end{equation*}

We want to find functions $c_k$ such that the boundary relations~\eqref{bc} are fulfilled for~$\Tilde{G_0}$:
\begin{equation}\label{kakoy}
{\cal W}(z)\cdot 
\begin{pmatrix}
c_0(y,z)\\
\vdots \\
e^{-iz\rho^{j-1}y}c_{j-1}(y,z)\\
\vdots \\
e^{-iz\rho^{n-1}y}c_{n-1}(y,z)\\
\end{pmatrix}
=\sum\limits_{k=0}^{n-1} \rho^{k}e^{-izy\rho^k}\cdot
\begin{pmatrix}
e^{izb\rho^k}Q_0(iz\rho^k)\\
\vdots \\
e^{izb\rho^k}Q_j(iz\rho^k)\\
\vdots \\
e^{izb\rho^k}Q_{n-1}(iz\rho^k)\\
\end{pmatrix},
\end{equation}
where ${\cal W}(z)$ is a matrix containing the boundary values of the exponents:
\begin{equation*}
{\cal W}_{jk}(z)=e^{iz\rho^{k-1}a}P_{j-1}(iz\rho^{k-1})+e^{iz\rho^{k-1}b}Q_{j-1}(iz\rho^{k-1}),\quad j,k\in \{1,\dots,n\}.
\end{equation*}

We solve this linear equation using Cramer's rule:
\begin{equation*}
c_{\beta-1}(y,z)=\sum\limits_{\alpha=1}^{n} \rho^{\alpha-1} e^{izy(\rho^{\beta-1}-\rho^{\alpha-1})}\cdot \frac{\Delta_{\alpha,\beta}(z)}{\Delta(z)},
\end{equation*}
where $\Delta$ is the determinant of ${\cal W}$, $\Delta_{\alpha,\beta}$ is the determinant of a matrix that coincides with~${\cal W}$ but the column $\beta$ that is 
changed for the $\alpha$-th column from the sum on the right in (\ref{kakoy}). Note that this changed column contains only the second summand of the 
$\alpha$-th column of ${\cal W}$. 

Finally, the formula for the Green function is
\begin{equation}\label{GFF}
\Tilde{G_0}(x,y,z)=\Tilde{K_0}(x,y,z)-\frac{i}{nz^{n-1}} \sum\limits_{\alpha,\beta=1}^{n}\rho^{\alpha-1}e^{iz(\rho^{\beta-1}x-\rho^{\alpha-1}y)}\cdot 
\frac{\Delta_{\alpha,\beta}(z)}{\Delta(z)}.
\end{equation}

\subsection{Asymptotics of the Green function}\label{s1}

\begin{lem}\label{asy}
Set
$$
\Gamma_1 = \big\{w=e^{i\phi}\,:\,\phi\in\big(0, \frac{\pi}{n}\big)\big\};\qquad 
\Gamma_2=\big\{w=e^{i\phi}\,:\,\phi\in\big(\frac{\pi}{n},\frac{2\pi}{n}\big)\big\}.
$$ 
Then for every sequence $R_l \rightarrow +\infty$ such that $R_l$ is separated from $|\lambda^0_N|^{\frac{1}{n}}$ and for all $j \in \{0,\dots,n-1\}$ 
the function 
$$
R_l^{n-1-j}\cdot|(\Tilde{G_0})^{(j)}_x(x,y,R_l w)|
$$ 
is uniformly bounded on $[a,b]^2\times (\Gamma_1\cup \Gamma_2)$. Next, for every $x\in [a,b]$ one has  
$$
R_l^{n-1} \cdot \Tilde{G_0}(x,y, R_lw) \rightarrow 0,\qquad R_l\rightarrow +\infty
$$
for a.e. $y \in [a,b]$ and a.e. $w \in \Gamma_1\cup \Gamma_2$. Moreover, the convergence is uniform on $C\times J$ for 
arbitrary compact set $C\subset[a,b]^2$ separated from the corners and the diagonal $\{x=y\}$ and for 
arbitrary compact set $J \subset \Gamma_1\cup \Gamma_2$.
\end{lem}

In what follows, when we write some limit over $R$ tending to $+\infty$ we mean the limit over this sequence $R_l$.

We turn to the proof of Lemma~\ref{asy}. The first part of this lemma (uniform estimates for $\Tilde{G}_0$ and its derivatives) can be easily extracted 
from \cite[\S4]{Nay}. However, to prove convergence to zero, one has to do more work. The proof is nothing but a treatment of formula~\eqref{GFF}, we 
evaluate each summand on its own. However, different summands are estimated in a different way, so we have to deal with several cases.

Note that for $x<y$
\begin{equation}\label{GFF1}
R^{n-1} \cdot \Tilde{G_0}(x,y, Rw)=-\frac{i}{nw^{n-1}} \sum\limits_{\alpha,\beta=1}^{n}\rho^{\alpha-1}e^{iRw(\rho^{\beta-1}x-\rho^{\alpha-1}y)}\cdot 
\frac{\Delta_{\alpha,\beta}(Rw)}{\Delta(Rw)},
\end{equation}
while for $x\ge y$
\begin{multline}\label{GFF2}
R^{n-1}\Tilde{G}_0(x,y,Rw)=\frac{i}{nw^{n-1}}\sum\limits_{\alpha=1}^n \rho^{\alpha-1}e^{iRw\rho^{\alpha-1}(x-y)}
\left(1- \frac{\Delta_{\alpha,\alpha}(Rw)}{\Delta(Rw)} \right)\\ 
-\frac{i}{nw^{n-1}} \sum\limits_{\alpha \ne \beta}\rho^{\alpha-1}e^{iRw(\rho^{\beta-1}x-\rho^{\alpha-1}y)}\cdot \frac{\Delta_{\alpha,\beta}(Rw)}{\Delta(Rw)}.
\end{multline}
We begin with asymptotics of elements of the matrix ${\cal W}$. If $Re(iw\rho^{k-1})>0$, then
\begin{multline*}
{\cal W}_{jk}(Rw)=e^{iRw\rho^{k-1}b}(iRw\rho^{k-1})^{d_{j-1}}\\
\times\left(b_{j-1}+O\Big(\frac{1}{R}\Big)+ 
e^{iRw\rho^{k-1}(a-b)}\Big(a_{j-1}+O\Big(\frac{1}{R}\Big)\Big)\right)\\
=e^{iRw\rho^{k-1}b}(iRw\rho^{k-1})^{d_{j-1}}\cdot(b_{j-1}+o(1)), \qquad R \to +\infty.
\end{multline*}
If $Re(iw\rho^{k-1})<0$, then
\begin{multline*}
{\cal W}_{jk}(Rw)=e^{iRw\rho^{k-1}a}(iRw\rho^{k-1})^{d_{j-1}}\\
\times\left(a_{j-1}+O\Big(\frac{1}{R}\Big)+ 
e^{iRw\rho^{k-1}(b-a)}\Big(b_{j-1}+O\Big(\frac{1}{R}\Big)\Big)\right)\\
=e^{iRw\rho^{k-1}a}(iRw\rho^{k-1})^{d_{j-1}}\cdot(a_{j-1}+o(1)), \qquad R \to +\infty.
\end{multline*}
We note that the ``$O$'' estimates are uniform on $\Gamma_1 \cup \Gamma_2$ and the ``$o$'' estimates are uniform on $J$.

We come to the moment where the cases of odd and even $n$ differ. Consider the function 
$$
\nu(w)=
\begin{cases}
\nu_1=\left[\frac{n+1}{2}\right],& w \in \Gamma_1;\\
\nu_2=\left[\frac{n}{2}\right],& w \in \Gamma_2.\\
\end{cases}
$$
Note that if $n$ is even, then $\nu(w)=\frac{n}{2}$ for $w \in \Gamma_1\cup \Gamma_2$. If $n$ is odd, then $\nu(w)=\frac{n+1}{2}$ for $w \in \Gamma_1$ 
and $\nu(w)=\frac{n-1}{2}$ for $w \in \Gamma_2$. This number $\nu(w)$ is characterized by the following property: if $k\leqslant \nu(w)$, then $Re(iw\rho^{k-1})<0$, 
if $\nu(w)<k\leqslant n$, then $Re(iw\rho^{k-1})>0$. Thus, for $w \in \Gamma_1\cup \Gamma_2$ the inequality $Re(iw\rho^{k-1})<0$ holds for 
$k\in \{1,\dots,\nu(w)\}$. 

Next, we write the asymptotics of determinant $\Delta$. 
We introduce the function 
\begin{equation}\label{ff}
\ff(Rw)=\sum\limits_{k=1}^{\nu(w)}|e^{iRw\rho^{k-1}(b-a)}|+\sum\limits_{k=\nu(w)+1}^{n}|e^{-iRw\rho^{k-1}(b-a)}|, \qquad w\in \Gamma_1 \cup \Gamma_2.
\end{equation}
Clearly, $\ff(Rw)\to 0$ uniformly on compact subsets of $\Gamma_1\cup \Gamma_2$ as $R \to +\infty$.

We factorize common factors from each column and row of $\Delta$ and get (see \cite[\S4]{Nay})
$$
\Delta(Rw)=e^{iaRw\sum\limits_{k=1}^\nu \rho^{k-1}+ibRw\sum\limits_{k=\nu+1}^{n} \rho^{k-1}}\!\cdot(iRw)^{\sum\limits_{j=0}^{n-1} d_j}\!\cdot\DD(Rw),
$$
where 
$$
\DD(Rw)=\hat\Delta+O\Big(\frac{1}{R}\Big)+O\Big(\ff(Rw)\Big)=\hat\Delta+o(1),\quad R \to +\infty,
$$
while $\nu=\nu_\kappa$ and $\hat\Delta=\hat\Delta^{[\kappa]}\equiv\det\hat {\cal W}^{[\kappa]}$ for $w \in \Gamma_\kappa$. Here the ``$O$'' estimates are uniform for 
$w\in \Gamma_1 \cup \Gamma_2$ and the ``$o$'' is uniform for $w\in J$. Recall that the determinants $\hat\Delta$ are non-zero by the 
Birkhoff regularity condition. Moreover, since $(\lambda^0_N)^{\frac 1n}$ are zeros of $\Delta(z)$, the function $\DD(Rw)$ is separated from zero for 
$R=R_l$ and $w \in \Gamma_1\cup \Gamma_2$ by our choice of the sequence $R_l$.\medskip 

Now we can write the asymptotics of terms in (\ref{GFF1}) and (\ref{GFF2}). 

\paragraph{Case $1$: $\alpha = \beta\leqslant \nu$.}
We have, as $R \to \infty$,
$$
\Delta_{\alpha,\alpha}(Rw)= e^{iRw(b\rho^{\alpha-1}-a\rho^{\alpha-1})} e^{iaRw\sum\limits_{k=1}^\nu \rho^{k-1}+ibRw\sum\limits_{k=\nu+1}^{n}\rho^{k-1}}
\!\!\cdot(iRw)^{\sum\limits_{j=0}^{n-1} d_j}\!\cdot (\hat\Delta_{\alpha,\alpha}+o(1)).
$$
Here $\hat\Delta_{\alpha,\alpha}$ is the determinant of a matrix that differs from $\hat {\cal W}$ only in the $\alpha$-th column. Namely, there are numbers 
$\rho^{(\alpha-1)d_j} b_{j}$ instead of $\rho^{(\alpha-1)d_j} a_{j}$. Thus, we obtain
$$
\frac{\Delta_{\alpha, \alpha}(Rw)}{\Delta(Rw)}=e^{iRw(b\rho^{\alpha-1}-a\rho^{\alpha-1})} \Big(\frac{\hat\Delta_{\alpha,\alpha}}{\hat\Delta}+o(1)\Big), 
\qquad R \to \infty.
$$
For $x<y$ this implies
\begin{equation*}
e^{iRw(\rho^{\alpha-1}x-\rho^{\alpha-1}y)}\,\frac{\Delta_{\alpha,\alpha}(Rw)}{\Delta(Rw)} = O(e^{iRw\rho^{\alpha-1}(b-a+x-y)})= o(1), \qquad R \to \infty,
\end{equation*}
if $(x,y)\ne (a,b)$. For $x\ge y$ we obtain, as $R \to +\infty$,
$$
e^{iRw\rho^{\alpha-1}(x-y)}\left( 1- \frac{\Delta_{\alpha,\alpha}(Rw)}{\Delta(Rw)} \right)=e^{iRw\rho^{\alpha-1}(x-y)}+
O(e^{iRw\rho^{\alpha - 1}(b - a +x-y)})=o(1),
$$
if $x\ne y$. Here the ``$o$'' estimates are uniform for $(x,y,w) \in C\times J$.

\paragraph{Case $2$: $\alpha = \beta > \nu$.}
We consider $\Delta-\Delta_{\alpha,\alpha}$ and use linearity of the determinant with respect to the $\alpha$-th column to get 
$e^{iRw\rho^{\alpha-1}a}P_{j-1}(iRw\rho^{\alpha-1})$ in the $\alpha$-th column. Using the same asymptotic formulas, we obtain
\begin{equation*}
e^{iRw\rho^{\alpha-1}(x-y)}\,\frac{\Delta(Rw) - \Delta_{\alpha,\alpha}(Rw)}{\Delta(Rw)} =  O(e^{iRw\rho^{\alpha-1}(a-b +x-y)}), \quad R \to \infty.
\end{equation*}
For $x \geq y$ this implies
$$
e^{iRw\rho^{\alpha-1}(x-y)}\left( 1- \frac{\Delta_{\alpha,\alpha}(Rw)}{\Delta(Rw)} \right)=O(e^{iRw\rho^{\alpha-1}(a-b +x-y)})=o(1), \quad R \to +\infty,
$$
if $(x,y)\ne (b,a)$. For $x<y$ we obtain, as $R \to +\infty$,
\begin{equation*}
e^{iRw(\rho^{\alpha-1}x-\rho^{\alpha-1}y)}\,\frac{\Delta_{\alpha,\alpha}(Rw)}{\Delta(Rw)} = - e^{iRw\rho^{\alpha-1}(x-y)} +
O(e^{iRw\rho^{\alpha-1}(b-a+x-y)})= o(1).
\end{equation*}
Here the ``$o$'' estimates are uniform for $(x,y,w) \in C\times J$.

\paragraph{Case $3$: $\alpha \ne \beta$.}
In this case we either directly use the same asymptotic formulas (but with the ``$O$'' estimates) or subtract the $\alpha$-th column from the $\beta$-th one 
in $\Delta_{\alpha,\beta}$ to make the exponent in the $\beta$-th column smaller (our choice of the procedure depends on the sign of 
$Re(iaw\rho^{\alpha-1})$).

\subparagraph{Subcase $3.1$: $\alpha,\beta \leqslant \nu$.}
In this case $Re(iaw\rho^{\alpha-1})<0$, so we directly use asymptotic formulas and get
\begin{multline}\label{<<}
\frac{\Delta_{\alpha,\beta}(Rw)}{\Delta(Rw)} = e^{iRw(b\rho^{\alpha - 1} - a\rho^{\beta - 1})}
\bigg(\frac{\hat\Delta_{\alpha,\beta}+O(\frac{1}{R})+O(\ff(Rw))}{\DD(Rw)}\bigg)\\
=e^{iRw(b\rho^{\alpha-1} - a\rho^{\beta-1})}\bigg(\frac{\hat\Delta_{\alpha,\beta}}{\hat\Delta}
+O\Big(\frac{1}{R}\Big)+O\Big(\ff(Rw)\Big)\bigg), \quad R \to +\infty.
\end{multline}
Here $\hat\Delta_{\alpha,\beta}$ is the determinant of a matrix that resembles $\hat {\cal W}$, the only difference is that there are numbers $\rho^{(\alpha-1)d_j} b_{j}$ instead of 
$\rho^{(\beta-1)d_j} a_{j}$ in the $\beta$-th column. The last equation in (\ref{<<}) holds because the denominator $\DD$ is separated from zero. 
The ``$O$'' estimates are uniform for $w \in \Gamma_1 \cup \Gamma_2$.

\subparagraph{Subcase $3.2$: $\alpha \leqslant \nu < \beta$.}
In this case $Re(iaw\rho^{\alpha-1})<0$ again, so we directly use asymptotic formulas and get
\begin{equation}\label{<>}
\frac{\Delta_{\alpha,\beta}(Rw)}{\Delta(Rw)} = e^{iRw(b\rho^{\alpha - 1} - b\rho^{\beta - 1})}\bigg(\frac{\hat\Delta_{\alpha,\beta}}{\hat\Delta}
+O\Big(\frac{1}{R}\Big)+O\Big(\ff(Rw)\Big)\bigg), \quad R \to +\infty.
\end{equation}
Here $\hat\Delta_{\alpha,\beta}$ is the determinant of a matrix that resembles $\hat {\cal W}$, the only difference is that there are numbers $\rho^{(\alpha-1)d_j} b_{j}$ instead of 
$\rho^{(\beta-1)d_j} b_{j}$ in the $\beta$-th column. The ``$O$'' estimates are uniform for $w \in \Gamma_1 \cup \Gamma_2$.

\subparagraph{Subcase $3.3$: $\alpha$,$\beta > \nu$.}
In this case $Re(iaw\rho^{\alpha-1})>0$, so we subtract the $\alpha$-th column from the $\beta$-th one in $\Delta_{\alpha,\beta}$. Arguing the same way 
as before, one gets
\begin{equation}\label{>>}
\frac{\Delta_{\alpha,\beta}(Rw)}{\Delta(Rw)} = e^{iRw(a\rho^{\alpha - 1} - b\rho^{\beta - 1})}\bigg(-\frac{\hat\Delta_{\alpha,\beta}}{\hat\Delta}
+O\Big(\frac{1}{R}\Big)+O\Big(\ff(Rw)\Big)\bigg), \quad R \to +\infty.
\end{equation}
Here $\hat\Delta_{\alpha,\beta}$ is the determinant of a matrix that resembles $\hat {\cal W}$, the only difference is that there are numbers $\rho^{(\alpha-1)d_j} a_{j}$ instead of 
$\rho^{(\beta-1)d_j} b_{j}$ in the $\beta$-th column. The ``$O$'' estimates are uniform for $w \in \Gamma_1 \cup \Gamma_2$.

\subparagraph{Subcase $3.4$: $\alpha > \nu \geqslant \beta$.}
In this case $Re(iaw\rho^{\alpha-1})>0$ again, so we subtract the $\alpha$-th column from the $\beta$-th one in $\Delta_{\alpha,\beta}$. Arguing the same 
way as before, one gets
\begin{equation}\label{><}
\frac{\Delta_{\alpha,\beta}(Rw)}{\Delta(Rw)} = e^{iRw(a\rho^{\alpha - 1} - a\rho^{\beta - 1})}\bigg(-\frac{\hat\Delta_{\alpha,\beta}}{\hat\Delta}
+O\Big(\frac{1}{R}\Big)+O\Big(\ff(Rw)\Big)\bigg), \quad R \to +\infty.
\end{equation}
Here $\hat\Delta_{\alpha,\beta}$ is the determinant of a matrix that resembles $\hat {\cal W}$, the only difference is that there are numbers $\rho^{(\alpha-1)d_j} a_{j}$ instead of 
$\rho^{(\beta-1)d_j} a_{j}$ in the $\beta$-th column. The ``$O$'' estimates are uniform for $w \in \Gamma_1 \cup \Gamma_2$.\medskip

In all subcases we obtain
\begin{equation*}
e^{iRw(\rho^{\beta-1}x-\rho^{\alpha-1}y)}\,\frac{\Delta_{\alpha,\beta}(Rw)}{\Delta(Rw)}=o(1), \quad R \to \infty,
\end{equation*}
if $(x,y) \notin \{(a,a),(a,b),(b,a),(b,b)\}$. Here the ``$o$'' estimates are uniform in $(x,y,w) \in C\times J$.\medskip 

Summing up the estimates of cases 1-3, we complete the proof of Lemma~\ref{asy}.\hfill$\square$

\begin{remark}
We note that for odd $n$ the numbers $\hat\Delta$ and $\hat\Delta_{\alpha,\beta}$ defined in the proof of Lemma~\ref{asy} depend on $w$ since 
the number $\nu$ depends on $w$. But these numbers are constants on $\Gamma_1$ and $\Gamma_2$. For even $n$ these numbers are constants on 
$\Gamma_1\cup \Gamma_2$.
\end{remark}

\subsection{Truncation of operator}

In this subsection we prove Theorem~\ref{Spint}.
We write down an identity
\begin{equation}
\label{K-G}
(G - G_0)(x,y,\lambda)=-\int\limits_a^b G_0(x,t,\lambda) \sum\limits_{k=0}^{n-2} p_k(t) G^{(k)}_t(t,y,\lambda)\, dt,
\end{equation}
where $p_k$ are the lower order coefficients of ${\mathbb L}$. It is a reformulation of Hilbert identity for resolvents, 
$$
\frac{1}{{\mathbb L} - \lambda} - \frac{1}{{\mathbb L}_0 - \lambda} = \frac{1}{{\mathbb L} - \lambda}\,({\mathbb L}_0 - {\mathbb L})\,\frac{1}{{\mathbb L}_0 - \lambda},
$$
in terms of Green functions. 

We differentiate equation \eqref{K-G} $j$ times with respect to $x$:
\begin{equation}
\label{k'}
G^{(j)}_x(x,y,\lambda)=(G_0)^{(j)}_x(x,y,\lambda)-\int\limits_a^b (G_0)^{(j)}_x(x,t,\lambda) \sum\limits_{k=0}^{n-2} p_k(t) G^{(k)}_t(t,y,\lambda)\,dt.
\end{equation}
Next, we multiply the expressions for $G^{(j)}_x$ by $p_j(x)$, sum them, and achieve 
\begin{multline*}
\sum\limits_{j=0}^{n-2} p_j(x) G^{(j)}_x(x,y,\lambda)=\sum_{j=0}^{n-2} p_j(x) (G_0^{(j)})_x(x,y,\lambda)\\
-\sum_{j=0}^{n-2} p_j(x) \int\limits_a^b (G_0)^{(j)}_x(x,t,\lambda) \sum\limits_{k=0}^{n-2} p_k(t) G^{(k)}_t(t,y,\lambda)\, dt.
\end{multline*}
Now let $|\lambda|^{\frac 1n}=R=R_l$ be taken from Lemma~\ref{asy}. 
Then the derivatives of $G_0$ can be estimated with the help of the first part of Lemma~\ref{asy}, and we obtain
$$
\bigg\|\sum\limits_{j=0}^{n-2} p_j(\cdot) G^{(j)}(\cdot,y,\lambda)\bigg\|_1\leqslant \frac{C}{|\lambda|^{\frac{1}{n}}}
+\frac{C}{|\lambda|^{\frac{1}{n}}}\cdot\bigg\|\sum\limits_{j=0}^{n-2} p_j(\cdot) G^{(j)}(\cdot,y,\lambda)\bigg\|_1.
$$
This implies
$$
\bigg\|\sum\limits_{j=0}^{n-2} p_j(\cdot) G^{(j)}(\cdot,y,\lambda)\bigg\|_1 \leqslant \frac{C}{|\lambda|^{\frac{1}{n}}}.
$$
We substitute this inequality into \eqref{k'} and get a pointwise estimate
\begin{equation}\label{grfest}
|G^{(j)}_x(x,y,\lambda)|\leqslant \frac{C}{|\lambda|^{\frac{n-1-j}{n}}}+\frac{C}{|\lambda|^{\frac{n-j}{n}}}\leqslant \frac{C}{|\lambda|^{\frac{n-1-j}{n}}}.
\end{equation}

Now we are ready to estimate the difference of the spectral functions of ${\mathbb L}$ and ${\mathbb L}_0$. Note that by formula~(\ref{K-G})
\begin{multline*}
\int\limits_{|\lambda|=R^{n}} |(G - G_0)(x,y,\lambda)|\, |d\lambda|\\
\leqslant
\int\limits_{\Gamma_1\cup \Gamma_2} \int\limits_a^b R^{n} |\Tilde{G_0}(x,t,Rw)|\cdot 
\Big|\sum\limits_{j=0}^{n-2} p_j(t) \Tilde{G}^{(j)}_t(t,y,Rw)\Big|\,dt |dw|.
\end{multline*}

By formula~(\ref{grfest}), the integrand has a majorant $M(t,w)= const\sum\limits_{j=0}^{n-2} |p_j(t)|$. We fix an $\eps>0$ and choose $\delta>0$ such 
that the integral of $M$ over the set of measure not more than $\delta$ is less than $\eps$. 

Next, we choose a compact set $C \subset [a,b]^2$, separated from the diagonal and the corners, such that the set $C_x=\{t\in [a,b]: (x,t) \notin C \}$ 
has measure not more than $\frac{\delta n}{2\pi}$ uniformly in $x$. Also we choose a compact set $J\subset\Gamma_1\cup \Gamma_2$ such that the measure of 
$\Gamma_1 \cup \Gamma_2 \setminus J$ is not more than $\frac{\delta}{b-a}$.

The integral over the set $([a,b]\setminus C_x)\times J$ tends to zero as $R\to\infty$ uniformly in $(x,y)\in [a,b]^2$, 
since by Lemma~\ref{asy} and formula~(\ref{grfest}) the integrand tends to zero uniformly on this set. The integral over the remaining set does not exceed 
$2\eps$. Thus, for $R$ large enough, the whole integral is not bigger than $3\eps$ for all $(x,y)\in [a,b]^2$, and the theorem follows.\hfill$\square$

\section{Proof of Theorems~\ref{gentraceth} and \ref{sep}}\label{3}

\subsection{Reduction to linear algebra}\label{3.1}
First of all, we can assume $\int\limits_a^b q(x) dx =0$ because adding a constant to $q$ only shifts the spectrum $\mu_N$, but does not change 
${\cal S}(q)$. We begin with a formula
\begin{equation}\label{inttrace}
\sum\limits_{|\lambda_N| < R^n}\lambda_N = -\frac{1}{2\pi i}\int\limits_{|\lambda| = R^n}\lambda\,\Sp\, \frac{1}{{\mathbb L} - \lambda}\, d\lambda,
\end{equation}
where the trace on the right is an integral operator trace 
$$
\Sp\,\frac{1}{{\mathbb L} - \lambda} = \int\limits_{a}^b G(x,x,\lambda)\, dx.
$$ 
Indeed, by the Lidskii theorem \cite{Lid},
\begin{equation*}
\sum_N\frac{1}{\lambda_N - \lambda} = \Sp\, \frac{1}{{\mathbb L} - \lambda}
\end{equation*}
for all $\lambda$ not in the spectrum of ${\mathbb L}$ (we use the fact that the resolvent $\frac{1}{{\mathbb L} - \lambda}$ is in the trace class, because $|\lambda_N|$ grows as 
$N^n$). We multiply this equation by $\lambda$, integrate over the circle $|\lambda| = R^n$, use the residue theorem and arrive at \eqref{inttrace}. 

Now we can express ${\cal S}(q)$ using the Hilbert identity for resolvents:
\begin{multline}\label{bigformula}
{\cal S}(q) = \frac{1}{2\pi i}\lim_{R \to \infty}\int\limits_{|\lambda| = R^n}
\lambda\,\Sp \, \Big(\frac{1}{{\mathbb L} - \lambda} - \frac{1}{{\mathbb L}+{\mathbb Q} - \lambda}\Big)\, d\lambda \\
= \frac{1}{2\pi i}\lim_{R \to \infty} \int\limits_{|\lambda| = R^n}\lambda \,\Sp \,\Big(\frac{1}{{\mathbb L}-\lambda}\,{\mathbb Q}\, \frac{1}{{\mathbb L}+{\mathbb Q}-\lambda}\Big)\,d\lambda\\
=-\frac{1}{2\pi i}\lim_{R \to \infty}\frac{1}{2}\int\limits_{|\lambda| = R^n}\lambda\,\Sp \,\Big(\Big(\frac{1}{{\mathbb L}-\lambda}
-\frac{1}{{\mathbb L}+{\mathbb Q}-\lambda}\Big)\,{\mathbb Q}\,\Big(\frac{1}{{\mathbb L}-\lambda}-\frac{1}{{\mathbb L}+{\mathbb Q}-\lambda}\Big)\Big)\, d\lambda\\
+ \frac{1}{2\pi i}\lim_{R \to \infty}\frac{1}{2}\int\limits_{|\lambda| = R^n}\lambda\,\Sp \,\Big(\frac{1}{{\mathbb L}-\lambda}\,{\mathbb Q} \,\frac{1}{{\mathbb L}-\lambda} 
+ \frac{1}{{\mathbb L}+{\mathbb Q}-\lambda}\,{\mathbb Q} \,\frac{1}{{\mathbb L}+{\mathbb Q}-\lambda}\Big)\,d\lambda.
\end{multline}
Obviously, we can take the limit over a sequence of $R$ separated from $|\lambda^0_N|^{\frac 1n}$.

We claim that the first integral in the right-hand side disappears at infinity. Indeed, it can be estimated as follows:
\begin{multline}\label{firstterm}
\int\limits_{|\lambda| = R^n}\lambda\,\Sp \,\Big(\Big(\frac{1}{{\mathbb L}-\lambda}
-\frac{1}{{\mathbb L}+{\mathbb Q}-\lambda}\Big)\,{\mathbb Q}\,\Big(\frac{1}{{\mathbb L}-\lambda}-\frac{1}{{\mathbb L}+{\mathbb Q}-\lambda}\Big)\Big)\, d\lambda\\
=\int\limits_{|\lambda| = R^n}\lambda\,\Sp \,\Big(\Big(\frac{1}{{\mathbb L}-\lambda}\,{\mathbb Q}\,
\frac{1}{{\mathbb L}+{\mathbb Q}-\lambda}\Big)\,{\mathbb Q}\,\Big(\frac{1}{{\mathbb L}-\lambda}\,{\mathbb Q}\,\frac{1}{{\mathbb L}+{\mathbb Q}-\lambda}\Big)\Big)\, d\lambda\\
 = O(R^{2 - \frac{4(n-1)}{n}})
\end{multline}
by inequality \eqref{grfest}. If $n > 2$, then this value tends to zero. In the remaining case we replace the first $\frac{1}{{\mathbb L}-\lambda}$ in 
(\ref{firstterm}) by $\frac{1}{{\mathbb L}_0 - \lambda}$. The difference tends to zero by Theorem~\ref{Spint} while the changed integral can be estimated 
with the help of the first part of Lemma~\ref{asy} and the Lebesgue Dominated Convergence theorem in the same way as we did at the end of the proof 
of Theorem~\ref{Spint}. Thus, the claim follows.

The second integral can be transformed as follows:
\begin{multline}\label{secondterm}
\int\limits_{|\lambda| = R^n}\lambda\,\Sp \,\Big(\frac{1}{{\mathbb L}-\lambda}\,{\mathbb Q} \,\frac{1}{{\mathbb L}-\lambda} 
+ \frac{1}{{\mathbb L}+{\mathbb Q}-\lambda}\,{\mathbb Q} \,\frac{1}{{\mathbb L}+{\mathbb Q}-\lambda}\Big)\,d\lambda \\
=\int\limits_{|\lambda| = R^n}\Sp\, \Big(\Big(\frac{\lambda}{({\mathbb L}-\lambda)^2}+\frac{\lambda}{({\mathbb L}+{\mathbb Q}-\lambda)^2}\Big)\,{\mathbb Q}\Big)\,d\lambda\\
=-\int\limits_{|\lambda| = R^n}\Sp\, \Big(\Big(\frac{1}{{\mathbb L}-\lambda} + \frac{1}{{\mathbb L}+{\mathbb Q}-\lambda}\Big)\,{\mathbb Q}\Big)\, d\lambda\\
=-2\int\limits_{|\lambda| = R^n}\Sp\, \Big(\frac{1}{{\mathbb L}_0-\lambda}\,{\mathbb Q}\Big)\, d\lambda+o(1),\quad R\to\infty.
\end{multline}
The first equality in (\ref{secondterm}) is identity $\Sp \,(ABC)=\Sp \,(BCA)$, the second one is integration by parts, and the third one 
follows from Theorem~\ref{Spint}. Thus, we arrive at
\begin{multline}\label{FunF}
{\cal S}(q)=-\frac{1}{2\pi i}\lim_{R \to \infty}\int\limits_{|\lambda| = R^n}\int\limits_a^b q(x)G_0(x,x,\lambda)\,dxd\lambda\\
=-\frac{1}{2\pi i}\lim\limits_{R \to \infty} \int\limits_{R(\Gamma_1\cup \Gamma_2)}\int\limits_a^b q(x)\Tilde{G_0}(x,x,z)nz^{n-1}dxdz
=\frac{1}{2\pi}\sum\limits_{\alpha,\beta=1}^n\I_{\alpha,\beta},
\end{multline}
where 
\begin{equation}\label{integrals}
\I_{\alpha,\beta}=\lim\limits_{R \to \infty} \int\limits_{R(\Gamma_1\cup \Gamma_2)}\int\limits_a^b
q(x) \rho^{\alpha-1}e^{izx(\rho^{\beta-1}-\rho^{\alpha-1})}\cdot \frac{\Delta_{\alpha,\beta}(z)}{\Delta(z)}\, dxdz.
\end{equation}
The last equality in~(\ref{FunF}) holds because of relation $\Tilde{K_0}(x,x,z)=0$.\medskip

If $\alpha=\beta$, the integral (\ref{integrals}) equals zero by the assumption $\int\limits_a^b q(x) dx =0$. So we turn to the case $\alpha \ne \beta$.
We use the asymptotic formulas for the quotients $\frac {\Delta_{\alpha,\beta}}{\Delta}$ obtained in the proof of Lemma~\ref{asy}. 

Denote by $\I^{[\kappa]}_{\alpha,\beta}$, $\kappa=1,2$, the same limit as $\I_{\alpha,\beta}$ but with the inner integral taken over 
$R\Gamma_\kappa$ instead of $R(\Gamma_1\cup\Gamma_2)$. Then $\I_{\alpha,\beta}=\I^{[1]}_{\alpha,\beta}+\I^{[2]}_{\alpha,\beta}$. There are four subcases. 

\subparagraph{Subcase $1$: $\alpha,\beta \leqslant \nu_\kappa$.}
We use (\ref{<<}) to write
\begin{multline}\label{eq100}
\I^{[\kappa]}_{\alpha,\beta}=
\frac{\hat\Delta^{[\kappa]}_{\alpha,\beta}}{\hat\Delta^{[\kappa]}}\cdot\rho^{\alpha-1}\lim\limits_{R \to \infty} \int\limits_{\Gamma_\kappa} 
\int\limits_a^b R q(x) e^{iRw(\rho^{\beta-1}(x-a)+(b-x)\rho^{\alpha-1})}\,dxdw\\
+\rho^{\alpha-1}\lim\limits_{R \to \infty} \int\limits_{\Gamma_\kappa} \big(O(1)+O(R\ff(Rw))\big) 
\int\limits_a^b q(x) e^{iRw(\rho^{\beta-1}(x-a)+(b-x)\rho^{\alpha-1})}\, dxdw.
\end{multline}
The last term here can be estimated as follows:
\begin{multline*}
\bigg|\int\limits_{\Gamma_\kappa} \big(O(1)+O(\ff(Rw))\big) 
\int\limits_a^b q(x) e^{iRw(\rho^{\beta-1}(x-a)+(b-x)\rho^{\alpha-1})}\, dxdw\bigg|
\leq\\
\sup\limits_{w \in \Gamma_\kappa} \bigg|\int\limits_a^b q(x) e^{iRw(\rho^{\beta-1}(x-a)+(b-x)\rho^{\alpha-1})} dx\bigg|\cdot
\int\limits_{\Gamma_\kappa} \big(O(1)+O(\ff(Rw))\big)\,|dw|.
\end{multline*}
The first factor tends to zero by Proposition~\ref{RL} as $R\to\infty$, while the second one is bounded by Proposition~\ref{bound} (see Appendix). Therefore, we obtain
\begin{equation}\label{eq101}
\I^{[\kappa]}_{\alpha,\beta}=
\frac{\hat\Delta^{[\kappa]}_{\alpha,\beta}}{\hat\Delta^{[\kappa]}}\cdot\rho^{\alpha-1}\lim\limits_{R \to \infty} \int\limits_{\Gamma_\kappa}
\int\limits_a^b R q(x) e^{iRw(\rho^{\beta-1}(x-a)+\rho^{\alpha-1}(b-x))}dxdw.
\end{equation}

The same calculations for three other subcases give the following formulas.
\subparagraph{Subcase $2$: $\alpha \leqslant \nu_\kappa < \beta$.}
\begin{equation}\label{eq102}
\I^{[\kappa]}_{\alpha,\beta}=\frac{\hat\Delta^{[\kappa]}_{\alpha,\beta}}{\hat\Delta^{[\kappa]}}\cdot\rho^{\alpha-1}\lim\limits_{R \to \infty} 
\int\limits_{\Gamma_\kappa}\int\limits_a^b R q(x) e^{iRw(\rho^{\alpha-1}-\rho^{\beta-1})(b-x)}\,dxdw.
\end{equation}

\subparagraph{Subcase $3$: $\alpha$, $\beta > \nu_\kappa$.}
\begin{equation}\label{eq103}
\I^{[\kappa]}_{\alpha,\beta}= -\frac{\hat\Delta^{[\kappa]}_{\alpha,\beta}}{\hat\Delta^{[\kappa]}}\cdot\rho^{\alpha-1}\lim\limits_{R \to \infty} 
\int\limits_{\Gamma_\kappa}\int\limits_a^b R q(x) e^{iRw(\rho^{\beta-1}(x-b)+\rho^{\alpha-1}(a-x))}\,dxdw.
\end{equation}

\subparagraph{Subcase $4$: $\alpha > \nu_\kappa \geqslant \beta$.}
\begin{equation}\label{eq104}
\I^{[\kappa]}_{\alpha,\beta}=-\frac{\hat\Delta^{[\kappa]}_{\alpha,\beta}}{\hat\Delta^{[\kappa]}}\cdot\rho^{\alpha-1}\lim\limits_{R \to \infty} 
\int\limits_{\Gamma_\kappa}\int\limits_a^b R q(x) e^{iRw(\rho^{\beta-1}-\rho^{\alpha-1})(x-a)}dwdx.
\end{equation}

In the subcase 1 we integrate with respect to $w$ and obtain
\begin{multline*}
\I^{[\kappa]}_{\alpha,\beta}=\frac{\hat\Delta^{[\kappa]}_{\alpha,\beta}}{\hat\Delta^{[\kappa]}}\cdot\rho^{\alpha-1}\\
\times\lim\limits_{R \to \infty} \int\limits_a^b q(x)\, 
\frac{e^{iR(\rho^{\beta-1}(x-a)+\rho^{\alpha-1}(b-x))(\sqrt{\rho})^\kappa}-e^{iR(\rho^{\beta-1}(x-a)+\rho^{\alpha-1}(b-x))
(\sqrt{\rho})^{\kappa-1}}}{i(\rho^{\beta-1}(x-a)+\rho^{\alpha-1}(b-x))}\,dx,
\end{multline*}
where $\sqrt{\rho}=e^{\frac{i\pi}{n}}$. Here the denominator is uniformly separated from zero, and the numerator is uniformly bounded. Thus, the integrand has a summable majorant $C|q(x)|$.
Moreover, since $\alpha\ne\beta$ and $\alpha,\beta \leqslant \nu_\kappa$, the numerator tends to zero for a.e. $x \in [a,b]$. By the
Lebesgue Dominated Convergence theorem, $I^{[\kappa]}_{\alpha,\beta}=0$. The same arguments show that $I^{[\kappa]}_{\alpha,\beta}=0$ in the subcase~3.

In subcases 2 and 4 after integration with respect to $w$ the denominators are not separated from zero. So, we should use the regularity of $q$ at the endpoints. 
Namely, under assumptions of Theorem~\ref{gentraceth} the functions $\psi_a$ and $\psi_b$ belong to $W^1_1([a,b])$, and
\begin{equation}\label{psi_ab}
q(x)=\psi_a(x)+(x-a)\psi'_a(x)=\psi_b(x)+(x-b)\psi'_b(x).
\end{equation}

Let us consider subcase 4. Using the first equality in (\ref{psi_ab}) we obtain
\begin{multline*}
\I^{[\kappa]}_{\alpha,\beta}=
-\frac{\hat\Delta^{[\kappa]}_{\alpha,\beta}}{\hat\Delta^{[\kappa]}}\cdot\rho^{\alpha-1}\lim\limits_{R \to \infty} \int\limits_{\Gamma_\kappa} 
\int\limits_a^b R \psi_a(x) e^{iRw(\rho^{\beta-1}-\rho^{\alpha-1})(x-a)}\,dxdw\\
-\frac{\hat\Delta^{[\kappa]}_{\alpha,\beta}}{\hat\Delta^{[\kappa]}}\cdot\rho^{\alpha-1}\lim\limits_{R \to \infty} \int\limits_a^b 
\psi'_a(x)\, \frac{e^{iR(\rho^{\beta-1}-\rho^{\alpha-1})(x-a)(\sqrt{\rho})^{\kappa}}-e^{iR(\rho^{\beta-1}-\rho^{\alpha-1})(x-a)(\sqrt{\rho})^{\kappa-1}}}
{i(\rho^{\beta-1}-\rho^{\alpha-1})}\,dx.
\end{multline*}
Since $\alpha > \nu_\kappa \geqslant \beta$, the last limit equals zero by Proposition~\ref{RL}. So, integrating by parts, we have
\begin{multline*}
\I^{[\kappa]}_{\alpha,\beta}=-\frac{\hat\Delta^{[\kappa]}_{\alpha,\beta}}{\hat\Delta^{[\kappa]}}\cdot\rho^{\alpha-1}\\
\times\lim\limits_{R \to \infty} \int\limits_{\Gamma_\kappa} 
\bigg[\psi_a(x)\,\frac{e^{iRw(\rho^{\beta-1}-\rho^{\alpha-1})(x-a)}}{iw(\rho^{\beta-1}-\rho^{\alpha-1})}\bigg|_a^b-
\int\limits_a^b \psi'_a(x)\, \frac{e^{iRw(\rho^{\beta-1}-\rho^{\alpha-1})(x-a)}}{iw(\rho^{\beta-1}-\rho^{\alpha-1})}\,dx\bigg]\,dw.
\end{multline*}
The last term here also tends to zero by Proposition~\ref{RL}. Moreover, the term with substitution $x=b$ tends to zero by the
Lebesgue Dominated Convergence theorem, and we arrive at
\begin{equation*}
\I^{[\kappa]}_{\alpha,\beta}=\frac{\hat\Delta^{[\kappa]}_{\alpha,\beta}}{\hat\Delta^{[\kappa]}}\cdot 
\frac{\rho^{\alpha-1}}{i(\rho^{\beta-1}-\rho^{\alpha-1})}\, \psi_a(a+)\cdot \int\limits_{\Gamma_\kappa}\frac{dw}{w}=
\frac{\pi}{n}\,\frac{\hat\Delta^{[\kappa]}_{\alpha,\beta}}{\hat\Delta^{[\kappa]}}\,\frac{\rho^{\alpha-1}}{\rho^{\beta-1}-\rho^{\alpha-1}}\,
\psi_a(a+).
\end{equation*}

By Cramer's rule, for all $\alpha > \nu_\kappa \geqslant \beta$ we have 
$$
\frac{\hat\Delta^{[\kappa]}_{\alpha,\beta}}{\hat\Delta^{[\kappa]}} = ((\hat {\cal W}^{[\kappa]})^{-1}{\cal A})_{\beta\alpha},
$$
and thus
\begin{equation*}
\I^{[\kappa]}_{\alpha,\beta}=\frac{\pi}{n}\,\psi_a(a+)\cdot
((\hat {\cal W}^{[\kappa]})^{-1}{\cal A})_{\beta\alpha} {\cal P}^{[\kappa]}_{\alpha\beta}, \qquad \beta \leqslant \nu_\kappa < \alpha,
\end{equation*}
where the matrix ${\cal P}^{[\kappa]}$ was introduced in (\ref{P}).

Since ${\cal P}^{[\kappa]}_{\alpha\beta}=0$ for other pairs $(\alpha,\beta)$, we obtain
\begin{equation}\label{eq105}
\sum\limits_{\alpha>\nu_\kappa\geq\beta} \I^{[\kappa]}_{\alpha,\beta}=\frac{\pi}{n}\,\psi_a(a+)\cdot\tr({\cal P}^{[\kappa]}(\hat {\cal W}^{[\kappa]})^{-1}{\cal A}).
\end{equation}

The same calculations for subcase 2 give
\begin{equation}\label{eq106}
\sum\limits_{\alpha\leq \nu_\kappa<\beta} \I^{[\kappa]}_{\alpha,\beta}=\frac{\pi}{n}\,\psi_b(b-)\cdot\tr({\cal Q}^{[\kappa]}(\hat {\cal W}^{[\kappa]})^{-1}{\cal B}).
\end{equation}


Since (\ref{FunF}) gives
\begin{equation*}\label{newF}
{\cal S}(q)=\frac{1}{2\pi}\sum\limits_{\alpha\ne\beta}\I_{\alpha,\beta}=\sum_{\kappa=1}^2 
\bigg(\sum_{\alpha>\nu_\kappa\geq \beta} \I^{[\kappa]}_{\alpha,\beta}+
\sum_{\alpha\leq \nu_\kappa< \beta} \I^{[\kappa]}_{\alpha,\beta}\bigg),
\end{equation*}
formula~(\ref{mainform}) follows immediately from (\ref{eq105}) and (\ref{eq106}). 

Equation (\ref{sumcoeff}) will be proved in the next subsection.

\subsection{Linear algebra calculations}\label{3.2}

In this subsection we skip index $\kappa$ for the sake of brevity.

\subsubsection{Proof of relation (\ref{sumcoeff})}

We begin with expanding ${\cal P}$ and ${\cal Q}$ into series. Consider two rows:
$$
\aligned
&v_k = (1, \rho^{k}, \rho^{2k} ,\dots, \rho^{(\nu-1)k}, 0,\dots,0);\\
&u_k = (0,\dots, 0, \rho^{\nu k}, \rho^{(\nu+1)k}, \dots, \rho^{(n-1)k}).
\endaligned
$$
Denote ${\cal P}_{(k)}=\bar{u}_k^T v_k$ and ${\cal Q}_{(k)}=\bar{v}^T_k u_k$. Then it is easy to verify that
$$
{\cal P}=-\lim\limits_{r\to 1-} \sum \limits_{k=0}^{\infty}r^k {\cal P}_{(k)};\qquad
{\cal Q}=-\lim\limits_{r\to 1-} \sum \limits_{k=0}^{\infty}r^k {\cal Q}_{(k)},
$$
and therefore
\begin{equation}\label{eq3}
\aligned
\tr({\cal P}\hat{{\cal W}}^{-1}{\cal A})= -\lim\limits_{r\to 1-} \sum \limits_{k=0}^{\infty}r^k \tr({\cal P}_{(k)} \hat{{\cal W}}^{-1}{\cal A});\\
\tr({\cal Q}\hat{{\cal W}}^{-1}{\cal B})= -\lim\limits_{r\to 1-} \sum \limits_{k=0}^{\infty}r^k \tr({\cal Q}_{(k)} \hat{{\cal W}}^{-1}{\cal B}).
\endaligned
\end{equation}

For any $k\in\mathbb{Z}$ and $ j\in\{0,1,\dots, n-1\}$ the direct calculation gives
$$
\aligned
\big({\cal A}\bar{u}_k^T\big)_{j+1} &=a_j(\rho^{\nu(d_j-k)}+\rho^{(\nu+1)(d_j-k)}+\dots+\rho^{(n-1)(d_j-k)});\\
%
%
\big(\hat{{\cal W}} \bar{v}_k^T\big)_{j+1}&=a_j(1+\rho^{d_j-k}+\rho^{2(d_j-k)}+\dots+\rho^{(\nu-1)(d_j-k)}).
\endaligned
$$
This implies
\begin{equation}\label{eq7}
{\cal A}\bar{u}_k^T+\hat{{\cal W}} \bar{v}_k^T=\sum\limits_{j=0}^{n-1} \sigma(k,d_{j})n a_j e_{j+1},
\end{equation}
where $e_j$ is $j$-th vector of standard basis, while
$$
\sigma(x,y)=\begin{cases}
1, & x\equiv y \pmod{n};\\ 0, & \mbox{otherwise.}
\end{cases}
$$

From (\ref{eq7}) we conclude that 
\begin{multline}\label{eq6}
\tr({\cal P}_{(k)} \hat{{\cal W}}^{-1}{\cal A})=\tr(v_k \hat{{\cal W}}^{-1}{\cal A} \bar{u}_k^T)=v_k\hat{{\cal W}}^{-1} {\cal A} \bar{u}_k^T\\
=- v_k\bar{v}_k^T+n\sum\limits_{j=0}^{n-1} \sigma(k,d_j) a_j v_k\hat{{\cal W}}^{-1}e_{j+1}
=-\nu+n\sum\limits_{j=0}^{n-1} \sigma(k,d_j) a_j v_k\hat{{\cal W}}^{-1}e_{j+1}.
\end{multline}
The same calculations 
give
\begin{equation}\label{eq1}
\tr({\cal Q}_{(k)}\hat{{\cal W}}^{-1}{\cal B})=-(n-\nu)+n\sum\limits_{j=0}^{n-1}\sigma(k,d_j)b_j u_k\hat{{\cal W}}^{-1}e_{j+1}.
\end{equation}

Since $\sigma(k,d_j)(a_jv_k+b_j u_k)=\sigma(k,d_j)e_{j+1}^T{\cal W}$, $j\in\{0,1,\dots,n-1\}$, formulas (\ref{eq3}), (\ref{eq6}) and (\ref{eq1}) imply
\begin{multline*}
\tr({\cal P}\hat{{\cal W}}^{-1}{\cal A})+\tr({\cal Q}\hat{{\cal W}}^{-1}{\cal B})=-\lim\limits_{r\to 1-}\sum\limits_{k=0}^\infty r^k
\Big(\tr({\cal P}_{(k)}\hat{{\cal W}}^{-1}{\cal A})+\tr({\cal Q}_{(k)}\hat{{\cal W}}^{-1}{\cal B})\Big)\\
=\lim\limits_{r\to 1-}\sum\limits_{k=0}^\infty r^k(n-n\sum\limits_{j=0}^{n-1}\sigma(k,d_j))
=\lim\limits_{r\to 1-}\Big(\frac{n}{1-r}-n\sum\limits_{j=0}^{n-1}\frac{r^{d_j}}{1-r^{n}}\Big)=\\
\lim\limits_{r\to 1-}\Big(\frac{n}{1-r}-\frac{n^2}{1-r^{n}}+n\sum\limits_{j=0}^{n-1}\frac{1-r^{d_j}}{1-r^{n}}\Big)=
\sum\limits_{j=0}^{n-1}d_j-\frac{n(n-1)}{2},
\end{multline*}
and (\ref{sumcoeff}) follows.

\subsubsection{Proof of relation (\ref{sep-even})}
Now we consider the case of almost separated boundary conditions. First, let $n=2m$. 

We introduce three sets: 
$$
I=\{k\geqslant 0 \ \colon k\equiv d_j \pmod{n} \ \ \mbox{for some}\ \  j<m \};
$$
$$
I_1=\{d_0,d_1,\dots,d_{m-1}\};\qquad I_2=\{0,\dots,2m-1\}\setminus I_1.
$$
For all $k\ge0$ the rows $v_k$ lie in the subspace $\Span\big\{e_{j+1}^T{\cal W}\,:\,j\in\{0,1,\dots,m-1\}\big\}$. Therefore, 
$v_k\hat{{\cal W}}^{-1}e_{j+1}=0$ for $j \geqslant m$.

If $k \in I$, then $k\equiv d_j \pmod{n}$ for a unique $j<m$. Hence $a_jv_k=e_{j+1}^T\hat{{\cal W}}$ and
$$
\sum\limits_{j=0}^{m-1} \sigma(k,d_j) a_j v_k\hat{{\cal W}}^{-1}e_{j+1}^T=1.
$$
Thus, by (\ref{eq6}), $\tr({\cal P}_{(k)} \hat{{\cal W}}^{-1}{\cal A})=m$ for $k\in I$. 

On the other hand, $\tr({\cal P}_{(k)} \hat{{\cal W}}^{-1}{\cal A})=-m$ for $k\notin I$, as $\sigma(k,d_j)=0$ for all $j<m$.

By (\ref{eq3}), we obtain
\begin{multline}\label{eq4}
\tr({\cal P} \hat{{\cal W}}^{-1}{\cal A})=-\lim\limits_{r\to 1-} \Big( \sum \limits_{k\in I} r^k m-\sum \limits_{0 \leqslant k \notin I} r^k m\Big)=
-m\lim\limits_{r\to 1-} \Big( \sum \limits_{k\in I} r^{k}-\sum \limits_{k\notin I} r^k \Big)\\
=-m\lim\limits_{r\to 1-} \Big( \sum \limits_{k\in I_1} \frac{r^{k}}{1-r^{2m}} - \sum \limits_{k\in I_2} \frac{r^{k}}{1-r^{2m}}\Big)=
m\lim\limits_{r\to 1-} \Big( \sum \limits_{k\in I_1} \frac{1-r^{k}}{1-r^{2m}} - \sum \limits_{k\in I_2} \frac{1-r^{k}}{1-r^{2m}}\Big)\\
=\frac{1}{2}\Big( \sum \limits_{k\in I_1}k -\sum \limits_{k\in I_2}k\Big)=
\sum\limits_{j=0}^{m-1}d_j-\frac{m(2m-1)}{2}.
\end{multline}
The same calculations for the second term in (\ref{mainform}) prove (\ref{sep-even}).

\subsubsection{Proof of relation (\ref{sep-odd})}

Now let $n=2m+1$. For $\kappa=2$ the previous arguments run almost without changing and give
\begin{equation*}
\tr({\cal P}^{[2]} (\hat{\cal W}^{[2]})^{-1}{\cal A})=
\sum\limits_{j=0}^{m-1}d_j-m^2.
\end{equation*}
The same calculations give
\begin{equation*}
\tr({\cal Q}^{[1]} (\hat{\cal W}^{[1]})^{-1}{\cal B})=
\sum\limits_{j=m+1}^{2m-1}d_j-m^2.
\end{equation*}
Substituting these formulas into (\ref{mainform}) and taking into account (\ref{sumcoeff}) we arrive at (\ref{sep-odd}).

\subsubsection{Proof of relation (\ref{periodic})}

Without loss of generality, we can assume that 
$$
a_j=1,\quad b_j=\vartheta, \quad d_j=j, \quad j \in \{0,\dots,n-1\}.$$
One can easily check that 
$$\hat{{\cal W}}^{-1}e_{j+1}^T=\frac{1}{n}
\left(1,\rho^{-j},\dots, \rho^{-(\nu-1)j}, \frac{1}{\vartheta}\rho^{-\nu j}, \dots, \frac{1}{\vartheta}\rho^{-(n-1)j}\right)^T,$$
so $\sigma(k,j)nv_k\hat{{\cal W}}^{-1}e_{j+1}^T=\sigma(k,j)\nu$. By (\ref{eq6}), for every $k \geqslant 0$ we have 
$$\tr({\cal P}_{(k)} \hat{{\cal W}}^{-1}{\cal A})=0.$$
Thus we obtain that $\tr({\cal P} \hat{{\cal W}}^{-1}{\cal A})=0$. Similarly, $\tr({\cal Q} \hat{{\cal W}}^{-1}{\cal B})=0$,
and (\ref{periodic}) follows.
 

\section{Appendix}

We need two technical statements. The first one is a variant of the Riemann--Lebesgue lemma.
\begin{prop}\label{RL}
Suppose $q \in L^1[a,b]$, $\Gamma \subset \{z \in \mathbb{C} \colon |z|=1\}$. Let $k_1, k_2\in \mathbb{C}$ satisfy $k_1 \ne 0$ and 
$Re(iw(k_1x+k_2))\leq 0$ for all $x \in [a,b]$ and $w \in \Gamma$. Then the following relation holds uniformly for $w \in \Gamma$.
$$
\int\limits_a^b q(x)e^{iRw(k_1x+k_2)}dx \to 0, \quad R \to +\infty.
$$
\end{prop}
\begin{proof}
Fix some $\eps>0$. Let a function $q_1 \in C^1([a,b])$ satisfy $q_1(a)=q_1(b)=0$ and $\int\limits_a^b |q-q_1|\leq \frac{\eps}{2}$. 
Then for $R$ large enough the following estimate holds:
$$
\left|\int\limits_a^b q_1(x)e^{iRw(k_1x+k_2)}dx\right|=\left|\frac{1}{iRwk_1}\int\limits_a^b q'_1(x)e^{iRw(k_1x+k_2)}dx\right|\leq 
\frac{1}{R|k_1|}\int\limits_a^b |q'_1|<\frac{\eps}{2}.
$$
Trivial estimate 
$$
\left|\int\limits_a^b (q(x)-q_1(x))e^{iRw(k_1x+k_2)}dx\right| \leq \int\limits_a^b |q-q_1|\leq \frac{\eps}{2}
$$
completes the proof.
\end{proof}

The second statement concerns the function $\ff(Rw)$ introduced by formula~(\ref{ff}).
\begin{prop}\label{bound}
There exists some constant $M>0$ such that for all $R>0$.
$$
\int\limits_{\Gamma_1 \cup \Gamma_2} R \ff(Rw)\, |dw| <M
$$
\end{prop}
\begin{proof}
We need to estimate several integrals of the same type. Most of them are exponentially small because the real part of the index is strictly less 
than zero on the whole arc $\Gamma_1 \cup \Gamma_2$. There are few integrals where the real part of the index tends to zero on the end of the arc. 
We write estimates for one of such integrals: 
$$
\int\limits_{\Gamma_1\cup \Gamma_2}R|e^{iRw(b-a)}|\,|dw| = 
\int\limits_0^{\frac{\pi}{n}}R e^{-R(b-a)\sin\phi}\,d\phi \leq
\int\limits_0^{\frac{\pi}{n}}R e^{-\frac{2}{\pi}R(b-a)\phi}\,d\phi< \frac {\pi}{2(b-a)}. 
$$
The other ones are estimated in the same way.
\end{proof}

\section{Acknowledgements}
The first author is supported by St.Petersburg State University grant 6.38.64.2012, 
the second author is supported by RFBR grant 11-01-00526, the second and the third authors are supported by Chebyshev Laboratory (SPbU), 
RF Government grant 11.G34.31.0026.

The authors are grateful to A. Minkin for his helpful advice and attracting our attention to the monograph \cite{Minkin}.


\begin{thebibliography}{AFT}

\bibitem{turkish}
A. Bayramov, Z. Oer, S. $\ddot{\OOO}$zt$\ddot{\UUU}$rk Uslu, and S. Kizilbudak \d{C}ali\d{s}kan,
{\em On the regularized trace of a fourth order regular differential equation}, Int. Journal of Contemp. Math. Sciences, {\bf 1} (2006), N5-8, 245--254.


\bibitem{GL} 
I.M. Gelfand, B.M. Levitan, {\em On a simple identity for the eigenvalues of a second-order differential operator}, DAN SSSR, {\bf 88} 
(1953), 593--596 (Russian). 

\bibitem{HK}
C.J.A. Halberg, Jr., V.A. Kramer, {\em A generalization of the trace concept}, Duke Math. J., {27} (1960), N4, 607--617.
 
\bibitem{KP1}
A.I. Kozko, A.S. Pechentsov, {\em Spectral function and regularized traces for singular differential operators of higher order}, Mat. Zam. {\bf 83} (2008), 
N1, 39--49 (Russian); English transl.: Math. Notes {\bf 83} (2008), N1-2, 37--47. 

\bibitem{KP3}
A.I. Kozko, A.S. Pechentsov, {\em Regularized traces of singular differential operators of order $2m$},
Sovr. Probl. Mat. Mech., V.3, 2009, MSU Publ., 45--57 (Russian).

\bibitem{Ko1}
A.G. Kostyuchenko, 
Sci.D. Dissertation, MSU, 1966 (Russian).

\bibitem{Ko2}
A.G. Kostyuchenko, {\em Asymptotic behavior of the spectral function of a singular
differential operator of order $2m$}, DAN SSSR, {\bf 168} (1966), N2, 276--279 (Russian); English transl.: 
Soviet Math. Dokl. {\bf 7} (1966), 632--635.

\bibitem{LS}
B.M. Levitan, I.S. Sargsyan, {\em Introduction to spectral theory: selfadjoint ordinary differential operators}, Moscow, Nauka, 1970 (Russian). 
English transl.: Translations of Mathematical Monographs, V.39, AMS, 1975.

\bibitem{Lid}
V.B. Lidskii, {\em Non-self-adjoint operators with a trace}, DAN SSSR, {\bf 125} (1959), N3, 485--487 (Russian).

\bibitem{Minkin}
A. Minkin, {\em Equiconvergence theorems for differential operators}, J. Math. Sci. (New York) {\bf 96} (1999), 3631--3715.

\bibitem{Nay}
M. A. Naimark, {\em Linear differential operators}, ed.2, Moscow, Nauka, 1969 (Russian). English transl. of the first ed.: 
{\em Linear Differential Operators: V.1, Elementary theory of linear differential operators}, Harrap, 1967.

\bibitem{PoSledam}
A.I. Nazarov, D.M. Stolyarov, P.B. Zatitskiy,  {\em Following traces of V.A. Sadovnichii}, Preprints of St.-Petersburg Math. Society, n.4, 2010 (Russian).

\bibitem{ZNS}
A.I. Nazarov, D.M. Stolyarov, P.B. Zatitskiy,  {\em On formula for regularized traces}, Dokl. RAN, {\bf 442} (2012), N2, 162--165 (Russian). 
English transl.: Doklady Mathematics, {\bf 85} (2012), N1, 29--32. 

\bibitem{SadPod}

V.A. Sadovnichii, V.E. Podolskii, {\em Traces of operators}, UMN, {\bf 61} (2012), N5, 89--156 (Russian). English transl.: 
Russian Math. Surveys (2006), 61(5):885--953.

\bibitem{Sad2}

V.A. Sadovnichii, {\em The trace of ordinary differential operators of high order}, Mat. Sbornik, {\bf 72(114)}, N2, 293--317 (Russian).
English transl.: Mathematics of the USSR-Sbornik {\bf 1} (1967), N2, 263--288.

\bibitem{SKP}
V.A. Sadovnichii, A.S. Pechentsov, A.I. Kozko, {\em Regularized traces of singular differential operators}, Dokl. RAN, {\bf 427} (2009), N4, 461--465 
(Russian). English transl.: Doklady Mathematics, {\bf 80} (2009), N1, 550--554. 

\bibitem{Sh}
R.F. Shevchenko, {\em On the trace of a differential operator}, DAN SSSR, {\bf 164} (1965), N1, 62--65 (Russian). English transl.: Soviet Math. Dokl., 
{\bf 6} (1965), 1183--1186.

\bibitem{Shk}
A.A. Shkalikov, {\em Boundary-value problems for ordinary differential equations with a parameter in the boundary conditions}, 
FAA, {\bf 16} (1982), N4, 324--326.

\bibitem{T}

J.D. Tamarkin, {\em On some general problems of the theory of ordinary linear differential
operators and on expansion of arbitrary function into serii}, Petrograd.
1917, 308 p.
\end{thebibliography}
\end{document}